\documentclass[10pt]{amsart}
\usepackage{amsmath}
\usepackage{amsfonts}
\usepackage{amssymb}
\usepackage{amsthm}
\usepackage{comment}
\usepackage{epsfig}
\usepackage{psfrag}
\usepackage{mathrsfs}
\usepackage{amscd}
\usepackage[all]{xy}
\usepackage{rotating}
\usepackage{lscape}
\usepackage{amsbsy}
\usepackage{verbatim}
\usepackage{moreverb}
\usepackage{sseq}
\usepackage{pdfpages}
\usepackage{enumerate}
\usepackage{float}
\restylefloat{table}

\usepackage{todonotes}
\usepackage{graphicx}
\usepackage{tikz-cd}
\usepackage{color}
\usepackage[colorlinks,linkcolor=blue,urlcolor=blue,citecolor=blue,destlabel=true]{hyperref}

\usepackage{cancel}
\usepackage[normalem]{ulem}

\usepackage[capitalise]{cleveref}
\numberwithin{equation}{section}

\newtheorem{thm}{Theorem}[section]
\newtheorem{theorem}{Theorem}[section]
\newtheorem{cor}{Corollary}[section]
\newtheorem{conj}{Conjecture}[section]

\newtheorem{prop}{Proposition}[section]

\newtheorem{lem}{Lemma}[section]
\newtheorem{lemma}{Lemma}[section]

\newtheorem*{thmintro1}{\Cref{lem:attaching}}
\newtheorem*{thmintro2}{\Cref{prop:6}}
\newtheorem*{thmintro3}{\Cref{thm:4}}
\newtheorem*{thmintro4}{\Cref{thm:3}}

\theoremstyle{definition}

\newtheorem{definition}{Definition}[section]

\newtheorem{notation}{Notation}[section]

\newtheorem{ex}{Example}[section]

\newtheorem{remark}{Remark}[section]

\makeatletter
\let\c@equation=\c@thm
\let\c@lem=\c@thm
\let\c@cor=\c@thm
\let\c@conj=\c@thm
\let\c@prop=\c@thm
\let\c@lem=\c@thm
\let\c@defn=\c@thm
\let\c@notation=\c@thm
\let\c@note=\c@thm
\let\c@exmp=\c@thm
\let\c@ex=\c@thm
\let\c@exmps=\c@thm
\let\c@rem=\c@thm
\let\c@warn=\c@thm
\let\c@claim=\c@thm
\let\c@convention=\c@thm
\let\c@conventions=\c@thm
\let\c@quest=\c@thm
\let\c@thmintro=\c@thm
\let\c@conjintro=\c@thm
\let\c@thmbig=\c@thm
\let\c@conbig=\c@thm
\let\c@facts=\c@thm
\let\c@definition=\c@thm
\let\c@proposition=\c@thm
\let\c@lemma=\c@thm
\let\c@remark=\c@thm
\let\c@conjecture=\c@thm
\let\c@theorem=\c@thm
\let\c@prediction=\c@thm
\let\c@guess=\c@thm
\makeatother

\newcommand{\id}{\mathrm{id}}

\DeclareMathOperator{\Aut}{Aut}

\DeclareMathOperator{\Hom}{Hom}

\DeclareMathOperator{\Iso}{Iso}

\newcommand{\Z}{\mathbb{Z}}
\newcommand{\Q}{\mathbb{Q}}

\newcommand{\R}{\mathbb{R}}

\newcommand{\cF}{\mathcal{F}}

\title{Homotopy Representations and the Picard Group of the Equivariant Stable Homotopy Category}

\author{Erik Knutsen}
\thanks{This material is based upon work supported by the National Science Foundation under Grant No. DMS 2143811.}

\begin{document}
\maketitle
\begin{abstract}
If $G$ is a finite group or a torus, it is known that there is an isomorphism between the Grothendieck group of homotopy representations  and that of generalized homotopy representations for $G$. We prove that there is such an isomorphism when $G$ is a compact Lie group with component group $\Gamma$ having the property that all projective $\Z\Gamma$-modules are stably free. This resolves a conjecture of Fausk, Lewis, and May for such $G$, giving a better description of the Picard group of the homotopy category of $G$-spectra.
\end{abstract}
\tableofcontents
\section{Introduction}
Let $G$ be a compact Lie group. Equivalence classes of invertible objects in the stable homotopy category of $G$-spectra form a group, $Pic(hSp^G)$, called the Picard group. Let $X$ be a generalized homotopy representation (see \cref{def:1}) and let $V$ be a finite-dimensional real orthogonal representation of $G$. A \emph{stable homotopy representation} is a $G$-spectrum of the form $\Sigma^{-V}\Sigma^{\infty}X$. Fausk, Lewis, and May showed that up to equivalence the invertible $G$-spectra are precisely the stable homotopy representations \cite{Fausk2001}. This was central to their construction of an exact sequence 
\begin{equation}\label{eqn:1}
    0\xrightarrow{} Pic(A(G))\xrightarrow{} Pic(hSp^G)\xrightarrow{} C(G),
\end{equation} 
where $C(G)$ is the group of continuous functions from the space of subgroups of $G$ to $\Z$ and $A(G)$ is the Burnside ring of $G$ with its Picard group is defined as the group of finitely generated, rank 1, projective $A(G)$-modules. Building on work of tom Dieck and Petrie, they further demonstrated that $Pic(hSp^G)$ is isomorphic to the Grothendieck group of generalized homotopy representations $V'(G)$. In contrast, the Grothendieck group for homotopy representations, $V(G)$, fits into a short exact sequence 
\begin{equation}\label{eqn:2}
    0\xrightarrow{} Pic(A(G))\xrightarrow{} V(G)\xrightarrow{} D(G)\xrightarrow{} 0,
\end{equation} where $D(G)$ is a particular subgroup of $C(G)$ described in \cite[Definition 2.1]{Bauer1989}.

The two exact sequences above naturally give rise to the following.
\begin{conj}[{\cite[Conjecture 4.5]{Fausk2001}}]\label{conj:1}
The canonical map $V(G)\xrightarrow{} V'(G)$ is an isomorphism for any compact Lie group $G$.
\end{conj}
When \cref{conj:1} holds, the short exact sequence \eqref{eqn:2} is a
refinement of the exact sequence \eqref{eqn:1}. It is already known that \cref{conj:1} holds for finite groups \cite{Dieck1982} and tori \cite{Dieck1982a}. The goal of our paper is to prove \cref{conj:1} in the case where $G$ is a compact Lie group such that its component group $\Gamma=G/G_0$ has the property that all projective $\Z\Gamma$-modules are stably free. If $G$ has this property, we will say that it is a \emph{stably free} compact Lie group. Then our primary result is 
\begin{thmintro4}
    Let $G$ be finite or stably free. 
    The homomorphism \[\rho:V(G)\xrightarrow{} V'(G)\] is an isomorphism.
\end{thmintro4}

Now we discuss further the content of the paper. In \cref{spaces}, we recall facts about $G$-spaces that are necessary for the definition of homotopy representations. Of particular importance are facts about isotropy subgroups of $G$ for a generalized homotopy representation $X$ since these determine the fixed points of $X$. In \cref{h.r.}, we define homotopy representations and two invariants: the dimension function and the degree function. These two invariants determine the $G$-homotopy classes of (generalized) homotopy representations. 

In \cref{attachinglemma}, we generalize a method due to tom Dieck and Petrie for attaching cells to certain $G$-CW-complexes to get  homotopy representations. In particular, we prove the following lemma. 
\begin{thmintro1}
Let $G$ be a compact $k$-dimensional Lie group that is finite or stably free. Let $Y$ be a generalized homotopy representation for $G$. Let $n(H)=n(Y)(H)$ be the dimension function for $Y$, with $n=n(\{e\})$.
Let $A$ be a finite $G$-CW-complex, and $f:A\xrightarrow{} Y$ be a $G$-map. Assume that the following hold:
\begin{enumerate}
    \item $n\ge k+3$.
    \item for $H\neq \{e\}$, \[A^H\simeq S^{n(H)}\] and \[\dim(A^H)=n(H),\]
    \item for each prime $p$, if $H$ is a nontrivial $p$-toral subgroup of $G$, $d(f)(H)$ is prime to $p$,
    \item $\dim A\le n-k-1$.
\end{enumerate}
Then there exists a homotopy representation $X$ obtained from $A$ by attaching cells of type $G_+\wedge D^i, i\le n-k$, and a $G$-map $F:X\xrightarrow{} Y$ extending $f$. If $G$ is finite, then $d(F)(\{e\})$ is prime to $|G|$, otherwise $d(F)(\{e\})=\pm 1$.
\end{thmintro1} 
\cref{lem:attaching} is pivotal in showing that, under stable homotopy equivalence, homotopy representations and generalized homotopy representations are indistinguishable.

In \cref{GHR}, we prove
\begin{thmintro2}
Let $G$ be a compact Lie group that is finite or stably free. Let $\Iso(Y)$ denote the set of isotropy subgroups of $Y$ and $w_H$ denote the dimension of $WH=NH/H$ for each closed subgroup $H\subseteq G$. Let $Y$ be a generalized homotopy representation such that: 
\begin{enumerate}[{(a)}]
\item for all $H\subseteq G$ closed, $n(H)\ge w_H+ 3$,
\item if $H\in \Iso(Y)$, then for all $K\supsetneq H$, \[n(H)\ge n(K)+w_H+1,\]
\item $\Iso(Y)$ is closed under intersections.
\end{enumerate}
Then there exists a homotopy representation $X$ and a $G$-homotopy equivalence $f\colon X \xrightarrow{} Y$.
\end{thmintro2}
This gives conditions under which a generalized homotopy representation $Y$ is $G$-homotopy equivalent to a homotopy representation by constructing a homotopy representation $X$ and a $G$-homotopy equivalence $f:X\xrightarrow{} Y$.

In \cref{Grothendieck}, we prove that, stably, every $G$-homotopy class of generalized homotopy representations contains a homotopy representation. 
\begin{thmintro3}
    Let $G$ be finite or stably free. Given a generalized homotopy representation $Y$ and a representation $V$ as in \cref{prop:8}, there exists a homotopy representation $X$ such that \[X\simeq_GY\wedge S^V.\]
\end{thmintro3}
Then we go on to use \cref{thm:4} to prove \cref{thm:3}, which resolves \cref{conj:1} for any stably free compact Lie group $G$.
In particular, \cref{conj:1} is true for connected compact Lie groups and for compact Lie groups such that the component group is cyclic of  order  $p$ for $p$ in the set \[\Lambda=\{2,3,5,7,11,13,17,19\}.\] This accounts for most of the classical examples of compact Lie groups.

\subsection*{Acknowledgements}
I would like to thank my advisor, Agn\`es Beaudry, for all of her support. This paper leans heavily on the work of tom Dieck, Petrie, and Bauer. Special thanks to Bauer and May for useful email exchanges.

\section{Background on Homotopy Representations}

\subsection{$G$-spaces}\label{spaces} 
We begin by gathering standard facts concerning $G$-spaces. We assume that our spaces $X$ are compactly generated and weak Hausdorff, that $G$ is a compact Lie group, and that the action of $G$ on $X$ is continuous.

\begin{definition}
Let $X$ be a based $G$-space. For $x\in X$, let \[Gx=\{gx\mid g\in G\}\] be the \emph{orbit of $x$}. Call \[X/G=\{Gx\mid x\in X\}\] the \emph{orbit space of $X$} where the topology on $X/G$ is the quotient topology. The \emph{orbit type} of $X$ is the set of isomorphism classes of orbits in $X/G$ as left $G$-sets. If this set is finite, then we say $X$ has \emph{finite orbit type.} 
\end{definition}

We restrict our attention to those subgroups that are also closed subspaces. Many of the properties that are relevant to us are invariant under conjugacy.

\begin{definition}
For a closed subgroup $H\subseteq G$, denote the conjugacy class of $H$ by $(H)$. Let \[\Psi(G)=\{(H)\mid H\text{ is a closed subgroup of }G\}.\] Let $NH$ be the normalizer of $H$ in $G$, and define $WH=NH/H$.
\end{definition}

Given two based $G$-spaces $X$ and $Y$, the smash product $X\wedge Y$ is a based $G$-space with $G$ acting diagonally.
The fixed points of a based $G$-space associated to a closed subgroup $H\subseteq G$ are defined by \[X^H=\{x\in X\mid hx=x, \text{ for all }h\in H\}.\] We denote the dimension of the fixed point space by $\dim(X^H)$.

Let $X$ and $Y$ be based $G$-spaces. 
We list the properties for fixed points that we need.

\begin{itemize}
        \item For conjugate subgroups $H$ and $K$, $X^H$ is homeomorphic to $X^K$.
        \item If $H\subseteq K$, then $X^K\subseteq X^H$.
        \item The fixed points of the smash product satisfy $(X\wedge Y)^H=X^H\wedge Y^H$.
        \item The fixed point subspace $X^H$ is a $WH$-space.
        \item For any $K/H\in WH$, $(X^H)^{K/H}=X^K$.
\end{itemize}
    
\begin{definition}
The \textit{isotropy group} of an element $x\in X$ is the set \[G_x=\{g\in G\mid gx=x\}.\] We denote the set of isotropy groups of a $G$-space $X$ by $\Iso(X)$. The \emph{isotropy type} of $X$ is the set of conjugacy classes of isotropy groups.
\end{definition}

Since the $G$-action is continuous and one point sets are closed in $X$, isotropy subgroups are closed.

Observe that for any $x\in X$, if every $h\in H$ fixes $x$, then $H\subseteq G_x$.  
Since $X$ is a based $G$-space, $G\in\Iso(X)$. Therefore there exist isotropy subgroups $\hat{H}\in \Iso(X)$ such that $H\subseteq \hat{H}$. 
It then follows that for any such $\hat{H}$, \[X^{\hat{H}}\subseteq X^H.\] 

If $H=g^{-1}Kg$ and $H=G_x$ then $K=G_{gx}$. So $\Iso(X)$ consists of complete conjugacy classes. 

Given two $G$-spaces $X$ and $Y$ the isotropy groups of $X\wedge Y$ relate to those of $X$ and $Y$ in the following way: \[\Iso(X\wedge Y)=\{G\} \cup\{H\cap K\mid H\in \Iso(X\backslash \{*\}), K\in \Iso(Y\backslash \{*\})\}.\]

If a based $G$-space $X$ has finite orbit type, then it has finite isotropy type. The relationship between isotropy type and orbit type is useful because our primary objects of interest have finite orbit type. The finiteness of the isotropy type is then fundamental to the proof techniques we will apply later. 
The next two propositions show that if the isotropy set of a $G$-space is closed under intersections, then the fixed points are completely determined by the isotropy groups.

\begin{prop}\label{prop:1}
Let $H\subseteq G$ and $X$ be a based $G$-space.
Suppose $\Iso(X)$ is closed under intersections. Then there is a unique minimal isotropy group $m(H)$ containing $H$. Furthermore, $X^H=X^{m(H)}$.
\end{prop}
\begin{proof}
    Let $m(H)$ be the intersection of all isotropy groups that contain $H$. Then $m(H)$ is minimal and unique. Since $H\subseteq m(H)$, $X^{m(H)}\subseteq X^H$. Suppose $x\in X^H.$ Since $m(H)\subseteq G_x$, it follows that $x\in X^{m(H)}$.
\end{proof}

\begin{prop}
    Let $X$ be a based $G$-space 
    with $\Iso(X)$ closed under intersections. Let $H,K\subseteq G$. If $K\in (H)$, then $m(K)\in (m(H))$.
\end{prop}
\begin{proof}
    Let $g\in G$ be such that $g^{-1}Kg=H$. Observe that \[m(H)\subseteq g^{-1}m(K)g\] because \[H\subseteq g^{-1}m(K)g\text{ and }g^{-1}m(K)g\in \Iso(X).\] Similarly, $m(K)\subseteq gm(H)g^{-1}$. So $m(K)\in (m(H))$.
\end{proof}

\subsection{Homotopy representations}\label{h.r.}
In this section we review generalized homotopy representations and the invariants that determine their $G$-homotopy types.

\begin{definition}\label{def:1}
A \emph{homotopy representation} for $G$ is a finite-dimensional based $G$-CW-complex $X$ of finite orbit type such that
\begin{enumerate}[(i)]
    \item for all closed subgroups $H\subseteq G$, there exists $n(X)(H)\in\Z_{\ge 0}$ such that \[X^H\simeq S^{n(X)(H)},\]
    \item $X^H$ is $n(X)(H)$-dimensional as a space.
\end{enumerate}
If we drop ii), then $X$ is a \emph{generalized homotopy representation}. If $X$ is a finite based $G$-CW-complex, then we call $X$ a \emph{finite homotopy representation}.
\end{definition}

\begin{remark}
    \cref{conj:1} is a statement concerning $G$-homotopy equivalence classes of (generalized) homotopy representations. Thus at times we may work with any $G$-space that has the $G$-homotopy type of a $G$-CW-complex. 
    This is consistent with how tom Dieck and Petrie originally defined homotopy representations in \cite{Dieck1982}. 
\end{remark}

Homotopy representations are a restricted class of generalized homotopy representations. As the more easily understood objects, it would be nice to be able to restrict our attention to these. To do so we want to be able to say when a $G$-homotopy class of generalized homotopy representations contains a representative that is a homotopy representation.
The two invariants, dimension and degree, allow us to determine the $G$-homotopy classes of generalized homotopy representations. 

The homotopy dimension of the fixed points associated to each subgroup is the first invariant and gives rise to a function defined on conjugacy classes of subgroups of $G$.

\begin{definition}
Given a generalized homotopy representation $X$, \emph{the homotopy dimension of $X$} is the function \[n(X):\Psi(G)\xrightarrow{}\Z_{\ge 0}\] for $n(X)(H)$ such that $X^H\simeq S^{n(X)(H)}$.
\end{definition}

It is important not to confuse this with the dimension of the underlying fixed point space $X^H$ which, as stated earlier, we denote by $\dim(X^H)$.

\begin{notation}
When the space $X$
in question is unambiguous, we will write:
\begin{itemize}
    \item $n(H)$ in place of $n(X)(H)$, and
    \item $n$ in place of $n(X)(\{e\})$.
\end{itemize}
\end{notation}

If $G$ is a nilpotent finite group or $G$ is a compact Lie group such that its connected component is abelian and its component group is nilpotent, the homotopy dimension function is sufficient to classify the $G$-homotopy types of generalized homotopy representations (see \cite{Bauer1988}). More generally,
in order to detect different $G$-homotopy types 
we also need the degree function.

In order to situate the definition of the degree in its proper context, we need to describe orientations. This is a technical necessity for generalized homotopy representations since the nontrivial cohomology groups associated to fixed point spaces are always $\Z$, and one needs to consider coherent choices of generators.

Let $X$ be a based $G$-CW-complex of finite orbit type, and let $CX$ denote the cone on $X$. Assume that for each closed subgroup $H\subseteq G$, \[H^{n(H)+1}(CX^H,X^H;\Z)=\Z,\] 
for some $n(H)\in\Z_{\ge0}$. For example, $n(H)$ could be the dimension of $X^H$ as a space or the homotopy dimension of a generalized homotopy representation.
The action of $WH$ on the pair $(CX^H,X^H)$ induces an action on this cohomology group and thus a homomorphism \[e_{X,H}:WH\xrightarrow{} \Aut(\Z).\]

\begin{definition}[{\cite[p. 69]{Dieck1987}}]\label{def:orientation}
 The homomorphism $e_{X,H}$ is called \emph{the orientation behaviour of $X$ at $H$.} An \emph{orientation} of $X$ is a choice of generator \[z_X(H)\in H^{n(H)+1}(CX^H,X^H;\Z)\] for each $H$. 
\end{definition}

Assume that for \[i:X^{m(H)}\xrightarrow{} X^H,\] $i^*(z_X(m(H))=z_X(H).$ Let \[l_g:X^H\xrightarrow{} X^K\] be left translation by $g$ where $K=gHg^{-1}$. We have \[l^*_gz_X(K)=e(g,H,X)z_X(H),\ e(g,H,X)\in\Z\backslash\{0\}.\]

\begin{definition}[{\cite[II.10.15]{Dieck1987}}]
    Let $X$ and $Y$ be generalized homotopy representations with the same dimension functions.  Orientations $z_X$ and $z_Y$ of $X$ and $Y$ are \emph{coherent} if $e(g,H,X)=e(g,H,Y)$ for all $H\subseteq G$.
\end{definition}

The following proposition states that we can suppress orientations on generalized homotopy representations because they are determined by the dimension function. Note that while tom Dieck worked with unbased spaces the proof of \cref{prop:5} depends on orientation and homotopy dimension which are identical in the based and unbased cases.

\begin{prop}[{\cite[II.10.16]{Dieck1987}}]\label{prop:5}
Suppose $X$ and $Y$ are generalized homotopy representations for $G$ with $n(X)=n(Y).$ Then
\begin{enumerate}
    \item $X$ and $Y$ have the same orientation behaviour,
    \item $X$ and $Y$ possess coherent orientations.
\end{enumerate}
\end{prop}

Let $f:X\xrightarrow{} Y$ be an equivariant map between oriented generalized homotopy representations for $G$ with $n(X)=n(Y).$ For every $H\subseteq G$, \[X^H\simeq Y^H\simeq S^{n(H)}.\] So the nonequivariant map $f^H$ has a  degree, $\deg(f^H)$, defined by \[z_X(H)=\deg(f^{H}) z_{Y}(H).\] 

\begin{definition}
Define \emph{the degree of the equivariant map $f$} to be a map \[d(f):\Psi(G)\xrightarrow{} \Z,\] such that \[d(f)(H)=\deg(f^H).\]
\end{definition}  

The dimension and degree functions determine the $G$-homotopy classes of (generalized) homotopy representations in the following way. 
\begin{prop}
   Let $G$ be a compact Lie group. Given two (generalized) homotopy representations, $X$ and $Y$, $X$ is $G$-homotopy equivalent to $Y$ if and only if $n(X)=n(Y)$ and there exists a map $f:X\xrightarrow{} Y$ such that $d(f)(H)=\pm 1$ for all $H\subseteq G$. 
\end{prop}
\begin{proof}
    A map $f:X\xrightarrow{} Y$ is a $G$-homotopy equivalence if and only if \[f^H:X^H\xrightarrow{} Y^H\] is a homotopy equivalence for all $H\subseteq G$. Observe that if $f^H$ is a homotopy equivalence then $n(X)(H)=n(Y)(H)$ and $d(f)(H)=\pm 1$. 
    
    Conversely, suppose $n(X)(H)=n(Y)(H)$ and $d(f)(H)=\pm 1$. Then \[\pi_{n(H)}(X^H)\cong \tilde{H}^{n(H)}(X^H)\text{ and }\pi_{n(H)}(Y^H)\cong \tilde{H}^{n(H)}(Y^H).\] Since $d(f)(H)=\pm 1$, the map \[f_*:\pi_{n(H)}(X^H)\xrightarrow{} \pi_{n(H)}(Y^H)\] is an isomorphism and $f^H$ is a homotopy equivalence.
    
\end{proof}

We want to determine which $G$-homotopy classes of generalized homotopy representations contain homotopy representations. We will first show under what conditions we can attach cells to a $G$-CW-complex in order to construct a homotopy representation $X$ that has the dimension function of a generalized homotopy representation $Y$ and a map $f:X\xrightarrow{} Y$. The resulting function may not be a homotopy equivalence, but it will be malleable enough that we can further alter $X$ to get a homotopy representation that is $G$-homotopy equivalent to $Y$.

\section{A Lemma for Attaching Cells}\label{attachinglemma}
Given a generalized homotopy representation $Y$, we want to be able to say when we can construct a homotopy representation $X$ with a $G$-map to $Y$. The capstone of this section is \cref{lem:attaching} which allows us to construct $X$ from an appropriate $G$-CW-complex by attaching cells provided that the component group of $G$ is suitable.

Let $G$ be a compact Lie group, $G_0$ its identity component, and let $\Gamma$ denote the component group $G/G_0$ of $G$. 

\begin{definition}\label{def:free}
Let $G$ be a compact $k$-dimensional Lie group. Let $X$ be a finite $G$-CW-complex. We say that $\tilde{H}_*(X)$ is $H_*(G)$\textit{-free (stably free, projective)} if all of the following hold:
\begin{enumerate}
    \item for some $n\ge 0$, $\tilde{H}_i(X)=0$ if $i<n$ or $i>n+k$,
    \item the map induced by the $G_0$ action on $X$, \[\eta:H_*(G_0)\otimes H_n(X)\xrightarrow{} \tilde{H}_*(X)\] is an isomorphism, and
    \item $H_n(X)$ is a free (stably free, projective) $\Z\Gamma$-module.
\end{enumerate}
\end{definition}

We want to verify that being $H_*(G)$-free according to the above definition does imply that $\tilde{H}_*(X)$ decomposes into a finite direct sum of copies of $H_*(G)$.

\begin{lem}\label{lem:1}
Let $G$ be a compact $k$-dimensional Lie group. Let $X$ be a finite based $G$-CW-complex such that $\tilde{H}_*(X)$ is a free $H_*(G)$-module. Then for some finite index set $I$, \[\tilde{H}_*(X)\cong \bigoplus_{I} H_*(G).\]
\end{lem}

\begin{proof}
    Since $X$ has finitely many cells and $G$ is a compact manifold, the homology $H_*(X)$ is finitely generated over $\Z$, hence finitely generated over $H_*(G)$. 
    By the K\"unneth isomorphism, 
    \[\mu:H_*(G_0)\otimes H_*(\Gamma)\xrightarrow{\cong} H_*(G) .\]
    The $\Gamma$ action induces an isomorphism of $\Gamma$-modules,
    \[H_0(G)\cong \Z\Gamma,\] and so an action of $\Z\Gamma$ on $H_n(X)$, where $n$ is as in (1) of \cref{def:free}.
    
    Let \[\{e_i:i\in I\}\] be a basis for the free $\Z\Gamma$-module $H_n(X)$. Then \[\{\eta(1\otimes e_i):i\in I\}\] 
    is an $H_*(G)$-basis for $\tilde{H}_*(X)$. 
\end{proof}

The following theorem of Oliver's, which we reprint, will be used in \cref{lem:attaching} to show when a $G$-CW-complex $X$ has $H_*(G)$-projective homology. We note that Oliver's proof covers both the based and unbased cases. A compact Lie group $H$ is called $p$-toral if $H_0$ is abelian and $H/H_0$ is a $p$-group.

\begin{theorem}[{\cite[Theorem 1]{Oliver1976}}]\label{thm:1} Let $G$ be a $k$-dimensional group. Assume $X$ is a finite $G$-CW-complex such that, for every prime $p$, if $H$ is a nontrivial $p$-toral subgroup of $G$, then $X^H$ is $\Z/p$-acyclic. If there exists $n\ge 0$ such that $\tilde{H}_*(X)=0$ except in dimensions $n$ to $n+k$ and $H_{n+k}(X)$ is $\Z$-free, then $\tilde{H}_*(X)$ is $H_*(G)$-projective.
\end{theorem}

We are now ready to prove our attaching lemma. In the finite case, this is due to tom Dieck and Petrie, \cite[Proposition 5.9]{Dieck1982}. A similar result is also proven by tom Dieck for tori, \cite[Proposition 10]{Dieck1982a}. We use ideas from the proofs of both propositions in proving \cref{lem:attaching}. 

Given an appropriate compact Lie group, if we have a $G$-CW-complex $A$ that satisfies the conditions for a homotopy representation for all subgroups of $G$ except the trivial subgroup and we have a $G$-map into a generalized homotopy representation $Y$, then the lemma enables us to attach free $G$-cells to $A$ to get a homotopy representation.
Recall here that $\dim(X)$ is the dimension of $X^{\{e\}}$ as a space.

\begin{lem}\label{lem:attaching}
Let $G$ be a compact $k$-dimensional Lie group such that if $k>0$, then for $\Gamma =G/G_0$ all $\Z\Gamma$-projective modules are stably free. Let $Y$ be a generalized homotopy representation for $G$. Let $n(H)=n(Y)(H)$ be the dimension function for $Y$, with $n=n(\{e\})$.
Let $A$ be a finite $G$-CW-complex, and $f:A\xrightarrow{} Y$ be a $G$-map. Assume that the following hold:
\begin{enumerate}
    \item $n\ge k+3$.
    \item for $H\neq \{e\}$, \[A^H\simeq S^{n(H)}\] and \[\dim(A^H)=n(H),\]
    \item for each prime $p$, if $H$ is a nontrivial $p$-toral subgroup of $G$, $d(f)(H)$ is prime to $p$,
    \item $\dim(A)\le n-k-1$.
\end{enumerate}
Then there exists a homotopy representation $X$ obtained from $A$ by attaching cells of type $G_+\wedge D^i, i\le n-k$, and a $G$-map $F:X\xrightarrow{} Y$ extending $f$. If $G$ is finite, then $d(F)(\{e\})$ is prime to $|G|$, otherwise $d(F)(\{e\})=\pm 1$.
\end{lem}

\begin{proof}
For $G$ finite this is  \cite[Proposition 5.9]{Dieck1982}.

Assume $k\ge 1$. By attaching cells of type $G_+\wedge D^i$, $i\le n-k-1$, to $A$ we get a space $B$ which is $(n-k-2)$-connected. By (4), $\dim A\le n-k-1$, so $\dim B= n-1$. From \[\pi_i(Y)=0\text{ for all }i\le n-1,\] we get a $G$-map \[g:B\xrightarrow{} Y\] that extends $f$ since the obstruction groups vanish and $B$ was constructed from $A$ using only free cells. Note that $g$ still satisfies (3), since
$B^H=A^H$. Let \[C_g=Y\cup_g(B\wedge I_+)/\sim\] be the mapping cone of the map underlying $g$. 
Since \[\tilde{H}_i(Y)=0\text{ for all }i\neq n\] and \[\tilde{H}_i(B)=0\text{ for }i\notin \{n-k-1,\dots,n-1\},\] the exact homology sequence \[\cdots\xrightarrow{}\tilde{H}_i(B)\xrightarrow{}\tilde{H}_i(Y)\xrightarrow{}\tilde{H}_i(C_g)\xrightarrow{\delta} \tilde{H}_{i-1}(B) \xrightarrow{}\cdots\] 
yields:
\begin{enumerate}[(a)]
    \item $\tilde{H}_{i+1}(C_g)\cong \tilde{H}_i(B)$ for $i\in \{n-k-1,\dots,n-2\}$,
    \item an exact sequence $0\xrightarrow{}\tilde{H}_{n}(Y)\xrightarrow{}\tilde{H}_{n}(C_g)\xrightarrow{\delta}\tilde{H}_{n-1}(B)\xrightarrow{} 0$,
    \item $\tilde{H}_i(C_g)=0$ for $i\notin \{n-k,\dots,n\}$.
\end{enumerate} The $G$-action on $C_g$ makes $\tilde{H}_*(C_g)$ into a graded $H_*(G)$-module.

Since $B$ is
$(n-1)$-dimensional, $\tilde{H}_{n-1}(B)$ is $\Z$-free. Thus $\tilde{H}_{n}(C_g)$ is $\Z$-free by (b). Now suppose that $H\subseteq G$ is a nontrivial $p$-toral subgroup for a prime $p$. Then \[\tilde{H}_*(C_g^H;\Z/p)\cong \tilde{H}_*(C_{g^H};\Z/p)=0\] because $g^H=f^H$ has degree prime to $p$ by (3). It now follows from \cite[Theorem 1]{Oliver1976} (see \cref{thm:1}) that $\tilde{H}_*(C_g)$ is a projective $H_*(G)$-module. 

Since $\tilde{H}_*(C_g)$ is a projective $H_*(G)$-module, by definition $\tilde{H}_{n-k}(C_g)$ is $\Z\Gamma$-projective and thus, by assumption, $\Z\Gamma$-stably free. That is, there exist finite-dimensional free $\Z\Gamma$-modules $F'$ and $F$ such that \[\tilde{H}_{n-k}(C_g)\oplus F'=F.\] Then by adding cells of type $G_+\wedge D^{n-k-1}$ to $B$ along nullhomotopic attaching maps \[S^{n-k-2}\xrightarrow{} B^{\{e\}},\] we can convert $\tilde{H}_{n-k}(C_g)$ to $F$. So we may assume that  $\tilde{H}_{n-k}(C_g)$ is $\Z\Gamma$-free. Therefore $\tilde{H}_*(C_g)$ is a free $H_*(G)$-module.

Choose a $H_*(G)$ basis $\{e_j\mid j\in J\}$ of $\tilde{H}_*(C_g)$. We construct a space $X$ from $B$ 
by attaching cells of type $G_+\wedge D^{n-k}$ to $B$ by using \[\delta(e_j)\in \tilde{H}_{n-k-1}(B)\cong \pi_{n-k-1}(B)\] as homotopy classes of attaching maps, noting that $n-k-1\ge 2$ by (1). The resulting space $X$ 
is $n$-dimensional. The map $g$ has an extension $F:X\xrightarrow{} Y$ since $Y$ is $(n-1)$-connected.

Let $j:B\xrightarrow{} X$ be inclusion. Then we get an induced cofibration sequence 
\begin{center}
    \begin{tikzcd}C_j\arrow[r,"\overline{F}"] & C_g\arrow[r,"\overline{j}"] & C_F\end{tikzcd}
\end{center} 
where \[\overline{F}:C_j=X\cup_j(B\wedge I_+)/\sim\ \xrightarrow{} C_g=Y\cup_g(B\wedge I_+)/\sim\]is $F$ on $X$ and the identity on $B\wedge I$ and \[\overline{j}:C_g=Y\cup_g(B\wedge I_+)/\sim\ \xrightarrow{} C_F=Y\cup_F(X\wedge I_+)/\sim\] is the identity on $Y$ and $j\wedge id$ on $B\wedge I$. This produces an exact homology sequence 
\begin{center}
    \begin{tikzcd}
    \tilde{H}_{i}(C_j)\arrow[r,"\overline{F}_*"] &
    \tilde{H}_{i}(C_g)\arrow[r,"\overline{j}_*"] &
    \tilde{H}_{i}(C_F)
    \end{tikzcd}
\end{center}

Since $X$ was constructed by attaching an ($n-k$)-cell for each $H_*(G)$-basis element of $\tilde{H}_*(C_g)$, $C_j$ is a wedge of copies of $G_+\wedge S^{n-k}$ indexed by the $H_*(G)$-basis $\{e_j\mid j\in J\}$.
Thus $\tilde{H}_{*}(C_j)$ is $H_*(G)$-free, and we have an isomorphism
\[H_{n-k}(C_j)\xrightarrow[\overline{F}_{*}]{\cong} H_{n-k}(C_g).\] 

This implies that $\overline{F}_*$ is an isomorphism in all degrees. Hence, $\tilde{H}_*(C_F)=0$. 
Thus, 
\[F_*\colon \tilde{H}_*(X) \xrightarrow{} \tilde{H}_*(Y)\] 
is an isomorphism.
Since $F$ is a map between simply connected spaces, it is a homotopy equivalence. Thus \[X^H\simeq S^{n(H)}\] and \[\dim(X^H)=n(H),\] by the homotopy equivalence for $H= \{e\}$ and by (2) for $H\neq \{e\}$. Therefore $X$ is a homotopy representation and $d(F)(\{e\})=\pm 1$.
\end{proof}

We end this section with a discussion concerning when the component group $\Gamma$ satisfies the condition that all $\Z\Gamma$-projectives are stably free. The condition that all $\Z\Gamma$-projective modules are stably free can be stated in terms of $K$-theory by requiring that $K_0(\Z\Gamma)\cong \Z$. Thus we include the following common definition.

\begin{definition}
Let $R$ be a unital ring. Denote the set of finitely-generated projective (left) $R$-modules by $\mathcal{P}(R)$. Define the equivalence relation $\sim$ on $\mathcal{P}(R)$ by $P\sim Q$ if there exist free (left) $R$-modules $F_1,F_2$ such that \[P\oplus F_1\cong Q\oplus F_2.\] The projective class group of $R$ is the abelian group \[K_0(R):=\mathcal{P}(R)/\sim.\]
\end{definition}

The compact Lie groups which satisfy the conditions of \cref{lem:attaching}  will be mentioned repeatedly through the rest of the paper. Their importance motivates the following definition.

\begin{definition}
    Let $G$ be a compact Lie group 
    such that the component group $\Gamma$ satisfies the condition that $K_0(\Z\Gamma)\cong \Z$. Then we call $G$ a \emph{stably free compact Lie group}.
\end{definition}

Finally, let us describe some examples of compact Lie groups that are stably free. In fact, all of these groups will have the slightly stronger condition that all of their projective $\Z\Gamma$-modules are free. If $G$ is connected, then it is stably free 
since all projective $\Z$-modules are free. \cref{thm:2} due to Rim tells us when $\Z\Gamma$-projectives are free for $\Gamma$ a cyclic group of prime order $p$. This, in conjunction with a result from number theory, yields a list of non-connected compact Lie groups that are stably free.

\begin{theorem}[{\cite[Theorem 6.24]{Rim1959}}]\label{thm:2}
Let $\Gamma$ be a cyclic group of prime order $p$. The projective class group $K_0(\Z\Gamma)$ is isomorphic to the ideal class group of the cyclotomic extension field $\Q(\zeta_p)$ of the rational number field. An element $P\in \mathcal{P}(\Z\Gamma)$ is free if and only if $[P]=0$ in $K_0(\Z\Gamma)$. 
\end{theorem}

\begin{remark}\label{rmk:3}
It has been shown (\cite[Theorem 11.1]{Washington1997}) that the ideal class group of $\Q(\zeta_p)$ for a prime $p$ is trivial if and only if $p$ is in the set \[\Lambda=\{2,3,5,7,11,13,17,19\}.\]
Thus $\Z\Gamma$-projective modules are stably free and \cref{lem:attaching} holds when $G$ is a connected compact Lie group or when $G$ is a compact Lie group such that $\Gamma$ is a cyclic group of order $p$ with $p\in \Lambda$. 
\end{remark}

\begin{ex}\label{orthogonal}
    For $n>0$, the orthogonal group $O(n)$ has component group $\Z/2$, so $O(n)$ is stably free.
\end{ex}
\begin{ex}\label{liegroups}
    The following groups are all connected and thus stably free:
    \begin{itemize}
        \item for $n>0$, the special orthogonal group $SO(n)$,
        \item for $n>2$, the spin group $\text{Spin}(n)$,
        \item for $n>0$, the unitary group $U(n)$,
        \item for $n>0$, the special unitary group $SU(n)$,
        \item for $n>0$, the compact symplectic group $\text{Sp}(n)$, and
        \item the compact forms of the five exceptional Lie groups: $G_2, F_4, E_6, E_7, E_8$.
    \end{itemize}
\end{ex}

\section{An Existence Lemma}\label{GHR}
In this section, we describe a set of conditions on a generalized homotopy representation that guarantee that its $G$-homotopy equivalence class contains a homotopy representation. In fact, we show a construction of such a homotopy representation using \cref{lem:attaching}, a result due to tom Dieck, and a couple of results due to tom Dieck and Petrie.

The first result we need is  \cite[Theorem 8.4.1]{Dieck1979}. It depends on a long list of assumptions which we describe here. 

Let $X$ be a finite-dimensional $G$-CW-complex of finite orbit type. Let $n(H)$ be the dimension of $X^H$ as a space. Assume the following:
\begin{enumerate}
    \item for $H\in \Iso(X)$, $n(H)>0$,
    \item for $H,K\in \Iso(X)$, if $H\subsetneq K$, then $n(H)>n(K)$,
    \item $H^{n(H)}(X^H;\Z)\cong \Z$.
\end{enumerate}
As in \cref{def:orientation}, the action of $WH$ on $X^H$ induces a homomorphism 
\[e_{X,H}:WH\xrightarrow{} \Aut(\Z)\] which we call the orientation behaviour of $X$ at $H$. Let $Z_{X,H}$
be the $WH$-module defined by $e_{X,H}$. Observe that \[(X^H,\bigcup_{K\supsetneq H} X^K)\xrightarrow{} (X^H/WH, (\bigcup_{K\supsetneq H} X^K)/WH)\] 
is a covering map, thus $Z_{X,H}$ can be interpreted as local coefficients on \[(X^H/WH, (\bigcup_{K\supsetneq H} X^K)/WH)\] and \[H^{n(H)}(X^H/WH, (\bigcup_{K\supsetneq H} X^K)/WH;Z_{X,H})\] arises from the cochain complex $\Hom_{WH}(C_{*}(X^H,\bigcup_{K\supsetneq H} X^K),Z_{X,H})$. Assume that if $WH$ is finite, then 
\[H^{n(H)}(X^H/WH, (\bigcup_{K\supsetneq H} X^K)/WH;Z_{X,H})\cong \Z.\]

Let $Y$ be a $G$-space. For $H\in \Iso(X)$, $Y^H$ is $(n(H)-1)$-connected and \[\pi_{n(H)}(Y^H)\cong \Z.\] It follows that $H^{n(H)}(Y^H)\cong \Z$ and we have an orientation behaviour \[e_{Y,H}:WH\xrightarrow{} \Aut(\Z)\] for $Y$ at $H$. Assume $e_{Y,H}=e_{X,H}$ for $H\in\Iso(X)$. We give $X$ an orientation by selecting generators for $H^{n(H)}(X^H)$ for every $H\subseteq G$. We orient $Y$ in the same way. Then for any map $f:X\xrightarrow{} Y$ the fixed point map $f^H$ has a well-defined degree $d(f)(H)=\deg(f^H)$.

\begin{prop}[{\cite[Theorem 8.4.1]{Dieck1979}}]\label{prop:Dieck}
Let $X$ be a finite-dimensional $G$-CW-complex of finite orbit type and $Y$ be a $G$-space. For $H\subseteq G$, let $n(H)=\dim(X^H)$ be the dimension of $X^H$ as a space. 

Assume the following:
\begin{enumerate}
    \item for $H\in \Iso(X)$, $n(H)>0$,
    \item for $H,K\in \Iso(X)$, if $H\subsetneq K$, then $n(H)>n(K)$,
    \item $H^{n(H)}(X^H)\cong \Z$,
    \item
    if $WH$ is finite, then \[H^{n(H)}(X^H/WH, (\bigcup_{K\supsetneq H} X^K)/WH;Z_{X,H})\cong \Z,\]
    \item for $H\in \Iso(X)$, $Y^H$ is $(n(H)-1)$-connected and \[\pi_{n(H)}(Y^H)\cong \Z,\]
    \item for $H\in \Iso(X)$, $e_{X,H}=e_{Y,H}$ and orientations have been chosen for $X$ and $Y$. 
\end{enumerate}
Under the assumptions above we have the following:
\begin{enumerate}
    \item[(a)] the equivariant homotopy set $[X, Y]_G$ is not empty,
    \item[(b)] elements $[f]\in[X, Y]_G$ are determined by the set \[\{d(f)(H)\mid H\in \Iso(X), |WH|<\infty\},\]
    \item[(c)] the value of each $d(f)(H)$ in the above set is  determined modulo $|WH|$ by the integers \[d(f)(K), \ \ K\supsetneq H\] and fixing these $d(f)(K)$ the possible $d(f)(H)$ fill the whole residue class mod $|WH|$.
\end{enumerate}
\end{prop}

There are a lot of conditions in the previous theorem, but it turns out that these are satisfied by a homotopy representation and a generalized homotopy representation which share a dimension function with the added conditions that $n(G)>0$ and for $H,K\in \Iso(X)$,\[H\subsetneq K\implies n(H)>n(K)+1.\]

In order to prove this we first include the following lemma. Note that (2') is a strengthening of (2).

\begin{lemma}[{\cite[Lemma 8.4.1]{Dieck1979}}]\label{lem:cohomology}
Let $X$ be a finite-dimensional $G$-CW-complex of finite orbit type. For $H\subseteq G$, let $n(H)=\dim(X^H)$ be the dimension of $X^H$ as a space. 

Assume the following:
\begin{enumerate}
    \item[(1)] for $H\in \Iso(X)$, $n(H)>0$,
    \item[(2')] for $H,K\in \Iso(X)$, if $H\subsetneq K$, then $n(H)>n(K)+1$,
    \item[(3)] $H^{n(H)}(X^H;\Z)\cong \Z$.
\end{enumerate} 
Then if $WH$ is finite, \[H^{n(H)}(X^H/WH, (\bigcup_{K\supsetneq H} X^K)/WH;Z_{X,H})\cong \Z.\]
\end{lemma}
\begin{proof}
    We include tom Dieck's proof, supplementing with details from \cite[Chapter II.4]{Dieck1979}. 

    From the assumption that $n(H)>n(K)+1$,
    we have that for $s\ge n(H)-1$, \[H^{s}((\bigcup_{K\supsetneq H} X^K)/WH;Z_{X,H})=0.\] From the long exact sequence associated to a pair, we get that \[H^{n(H)}(X^H/WH, (\bigcup_{K\supsetneq H} X^K)/WH;Z_{X,H})\cong H^{n(H)}(X^H/WH;Z_{X,H}).\] 
    Recall that $H^*(X^H/WH;Z_{X,H})$ arises from the cochain complex $\Hom_{WH}(C_{*}(X^H),Z_{X,H})$. Since $n(H)>n(K)+1$,  $C_{n(H)}(X^H)$ is a free $WH$-module. Therefore the norm map 
    \begin{align*}
        \Hom(C_{n(H)}(X^H),\Z)&\xrightarrow{N}\Hom_{WH}(C_{n(H)}(X^H),Z_{X,H}),\\ f(x)&\mapsto \sum_{w\in WH}wf(w^{-1}x)
    \end{align*}  
    is surjective. From this the map on cohomology \[H^{n(H)}(X^H;\Z)\xrightarrow{\phi} H^{n(H)}(X^H/WH;Z_{X,H})\] is surjective. Let $p:X^H\xrightarrow{} X^H/WH$, and let \[p^*:H^{n(H)}(X^H/WH;Z_{X,H})\xrightarrow{} H^{n(H)}(X^H;\Z)\] be the induced map on cohomology. Then $p^*\circ\phi$ is multiplication by $|WH|$. So it suffices to show that $H^{n(H)}(X^H/WH;Z_{X,H})$ is torsion free. The left exact sequence \[0\xrightarrow{} Z_{X,H}\xrightarrow{j} C_{n(H)}(X^H)\xrightarrow{\partial} C_{n(H)-1}(X^H)\] induces a sequence 
    \begin{align*}
        \Hom_{WH}(C_{n(H)-1}(X^H),Z_{X,H})&\xrightarrow{\delta_0} \Hom_{WH}(C_{n(H)}(X^H),Z_{X,H})\\
        &\xrightarrow{\delta_1} \Hom_{WH}(Z_{X,H},Z_{X,H})=\Z.
    \end{align*}
    We will show that $\ker \delta_1=\text{Im}\ \delta_0$. Suppose $\varphi\in\ker\delta_1$. Let \[\psi:C_{n(H)}(X^H)\xrightarrow{} C_{n(H)}(X^H)\] be a map with $N(\psi)=\id$. Then $\varphi\psi j\in \Hom_{WH}(Z_{X,H},Z_{X,H})$. Then \[|WH|\varphi\psi j=N(\varphi\psi j)=\varphi N(\psi)j=\varphi j=0.\] So $\varphi\psi j=0$ and there exists \[\varphi'\in \Hom_{WH}(C_{n(H)-1}(X^H),Z_{X,H})\] such that $\varphi\psi=\varphi'\partial$. However, when we apply the norm to make $\varphi\psi$ equivariant we have $N(\varphi\psi)=\varphi$, so $\varphi\in \text{Im }\delta_0$ and $H^{n(H)}(X^H/WH;Z_{X,H})=\Z$.
\end{proof}

\begin{lemma}\label{lem:conditions}
Let $G$ be a compact $k$-dimensional Lie group. Let $X$ be a homotopy representation and $Y$ a generalized homotopy representation with $n(X)=n(Y)$ such that $n(G)>0$ and for $H,K\in \Iso(X)$, \[H\subsetneq K\implies n(H)>n(K)+1.\] All conditions of \cref{prop:Dieck} are satisfied by $X$ and $Y$.
\end{lemma}
\begin{proof}
    Condition 1 holds since $n(G)>0$ and $n(H)\ge n(G)$ for all $H\subseteq G$. Condition 2 is assumed. Conditions 3 and 5 are direct consequences of the fact that \[X^H\simeq Y^H\simeq S^{n(H)}.\] Condition 4 is \cref{lem:cohomology}. Condition 6 is \cref{prop:5}.     
\end{proof}

The following corollary is a direct consequence of the previous two results. Given a map $f:X\xrightarrow{} Y$ between a homotopy representation and  a generalized homotopy representation that satisfy the assumptions of \cref{lem:conditions}, the following corollary states that for all isotropy subgroups $H$ with finite $WH$, given an integer $d$ congruent to $d(f)(H)$ modulo $|WH|$, we may assume without loss of generality that $d(f)(H)=d$.

\begin{cor}\label{prop:4}
Assume that $X$ is a homotopy representation and $Y$ is a generalized homotopy representation such that \[n(X)=n(Y).\] Let $f:X\xrightarrow{} Y$ be a $G$-map.  Assume $n(G)>0$. Suppose $H\in \Iso(X)$ such that
\begin{enumerate}
    \item $WH$ is finite, and
    \item for any $K\in \Iso(X)$ that properly contains $H$, $n(H)>n(K)+1$.
\end{enumerate} 
Then for every $d\in \Z$ such that \[d\equiv d(f)(H)\mod |WH|,\] there exists a $G$-map $g:X\xrightarrow{} Y$ such that $d(g)(H)=d$, and for every $K$ such that $K\supsetneq H$, $g\vert_{X^K}\simeq f\vert_{X^K}$.
\end{cor}
\begin{proof}
    By \cref{lem:conditions}, we can apply \cref{prop:Dieck}. \cref{prop:Dieck} states that the $G$-homotopy classes of maps are determined by the fixed point degrees for isotropy subgroups $L$ with finite $WL$. Assume that for each $K\supsetneq H$, $d(f)(K)$ has been fixed. This determines $d(f)(H)$ modulo $|WH|$, and for each $d\equiv d(f)(H)\mod |WH|$, there is a homotopy equivalence class $[g]$ with representative $g$ such that $d(g)(H)=d$. Furthermore, for each $K\supsetneq H$, $d(g)(K)=d(f)(K)$. Since these are maps between spheres, $g\vert_{X^K}\simeq f\vert_{X^K}$.
 \end{proof}

In \cref{prop:6}, we will construct a homotopy representation that is $G$-homotopy equivalent to a generalized homotopy representation by attaching $WH$-cells for each isotropy subgroup $H$. In order for this to work, we need to be able to extend the resulting map of $WH$-spaces to a map between $G$-spaces.

\begin{lem}[{\cite[Lemma 4.11]{Dieck1982}}]\label{lem:3}
Let \[k:A\xrightarrow{} Y\] be an equivariant map between finite-dimensional $G$-CW-complexes. For $H\subseteq G$, let $W$ be a finite-dimensional $WH$-CW-complex containing $A^H$ as a subcomplex. Let 
\[h:W\xrightarrow{} Y^H\]
be a $WH$-map extending $g^H$. If $WH$ acts freely on $W\backslash A^H$, then there is a unique $G$-CW-complex $X$ containing $A\cup W$ such that
\[X/A=(G_+\wedge_{NH}W)/(G_+\wedge_{NH}A^H),\]
and a $G$-map 
\[f:X\xrightarrow{} Y\] 
that extends $k$ and $h$. 
\end{lem}

\begin{remark}\label{rmk:2}
The space $X$ is constructed from $A$ using only $G/H$-cells with attaching maps that are the unique $G$-extensions of the attaching maps for $W$. Since $WH$ acts freely on $W\setminus A^H$, there are no fixed points added for subgroups $K\supsetneq H$. So for $K\supsetneq H$, \[X^K=A^K\text{ and }f^K=k^K.\] Each cell attached to $A$ to get $X$ is of the form $(G/NH)_+\wedge WH_+\wedge D^i$ where $WH_+\wedge D^i$ is a cell in $W$. So all $H$-fixed points of $X$ are contained in $W$, and \[X^H=W\text{ and } f^H=h.\]
\end{remark}

In the finite version of \cref{lem:attaching}, we get a homotopy representation $X$ with a $G$-map $F$ to a generalized homotopy representation such that $\deg(F^H)$ is prime to $|WH|$. The next lemma allows us to attach cells to $X$ and alter the map so that it has the degree we want.

\begin{lemma}[{\cite[Lemma 6.4]{Dieck1982}}]\label{lem:4}
    Let $G$ be a finite group. Let $Z$ be a $n$-dimensional $G$-CW-complex such that \[Z\simeq S^n, \text{ for some }n\ge 3.\]  Suppose $Z$ is obtained from its $(n-1)$-skeleton $Z_{n-1}$ by attaching cells of type $G_+\wedge D^n$. Let $l\in \Z$ be prime to $|G|$. 
    
    Then there exists an $n$-dimensional $G$-CW-complex $B$ obtained from $Z_{n-1}$ by attaching cells of type $G_+\wedge D^{n-1}, G_+\wedge D^n$ and a $G$-map $\varphi:B\xrightarrow{} Z$ such that:
    \begin{enumerate}
        \item $B\simeq S^n$,
        \item $\deg(\varphi)=l$,
        \item $\varphi\vert_{B_{n-2}}=id$ (note: $B_{n-2}=Z_{n-2})$.
    \end{enumerate}
\end{lemma}

\begin{remark}
    In \cref{lem:4}, if $Z$ is a homotopy representation then so is $B$ since there are only free orbits added and $B$ still satisfies the defining characteristics of a homotopy representation.
\end{remark}

With the following two definitions we will have all of the tools at our disposal to prove the main result of this section, \cref{prop:6} below. 

\begin{definition}
    Let $G$ be a compact Lie group. A \emph{closed family $\cF$ of subgroups of $G$} is a collection of closed subgroups such that if $H\in \cF$, then every $K$ containing $H$ is in $\cF$ and every conjugate of $H$ is in $\cF$.
\end{definition}

\begin{definition}
Let $G$ be a compact $k$-dimensional Lie group, for every $H\subseteq G$, let $w_H$ be the dimension of $WH$ as a compact Lie group.
\end{definition}

\begin{lem}\label{prop:6}
Let $G$ be a compact Lie group that is finite or stably free. Let $Y$ be a generalized homotopy representation such that: 
\begin{enumerate}[{(a)}]
\item for all $H\subseteq G$ closed, $n(H)\ge w_H+ 3$,
\item if $H\in \Iso(Y)$, then for all $K\supsetneq H$, \[n(H)\ge n(K)+w_H+1,\]
\item $\Iso(Y)$ is closed under intersections.
\end{enumerate}
Then there exists a homotopy representation $X$ and a $G$-homotopy equivalence $f\colon X \xrightarrow{} Y$.
\end{lem}

\begin{proof}
The proof proceeds inductively over orbit types of $Y$.  The initial step is to take the closed family $\cF=\{G\}$ and set $X(\cF)=S^{n(G)}$ with a trivial $G$-action.

Assume we have a closed family of subgroups $\cF$, a $G$-CW-complex $X(\cF)$, and a $G$-map $f_{\cF}:X(\cF)\xrightarrow{} Y$ such that:
\begin{enumerate}[{(i)}]
    \item $\Iso(X(\cF))=\cF\cap \Iso(Y)$,
    \item for $K\in \cF$, $X(\cF)^K\simeq Y^K$ and $\dim X(\cF)^K=n(K)$,
    \item for $K\in \cF$, $\deg(f_{\cF}^K)=\pm 1$.
\end{enumerate}

Note that since $\Iso(Y)$ is closed under intersections, for any subgroup $K\subseteq G$ there exists a minimal isotropy group containing $K$, (see \cref{prop:1}).
Take a maximal $H\in \Iso(Y)\setminus\cF$ and define 
\[\cF'=\cF\cup \{K\mid K\supseteq \Tilde{H}\text{ for some }\Tilde{H}\in(H)\}.\] 
We want to now construct $X(\cF')$ and an extension $f_{\cF'}$ of $f_{\cF}$ such that i.-iii. are satisfied. 

We will apply \cref{lem:attaching} to the $WH$-map \[f_{\cF}^H:X(\cF)^H\xrightarrow{} Y^H
.\] To do so we check the four conditions from \cref{lem:attaching} with $k=w_H$.
\begin{enumerate}
    \item By assumption, $n(H)\ge w_H+3$.
    \item For any $K/H\in WH$, with $K\supsetneq H$, so that $K\in\cF$, we get from (ii) that \[(X(\cF)^H)^{K/H}=X(\cF)^K\simeq S^{n(K)}\] and \[\dim((X(\cF)^H)^{K/H})=n(K).\]
    \item For any $K/H\in WH$, with $K\supsetneq H$, so that $K\in\cF$, we get from (iii)  that \[(\deg(f_{\cF}^H))^{K/H}=\deg(f_{\cF}^K)=\pm1,\] which implies that for each prime $p$, $(\deg(f_{\cF})^H)^{K/H}$ is prime to $p$ if $K/H$ is a nontrivial $p$-toral subgroup of $WH$.
    \item By assumption, \[n(H)\ge n(K)+w_H+1.\] So \[n(H)-w_H-1\ge n(K)=\dim(X(\cF)^K).\]
    Since 
    \[\dim(X(\cF))=\max\{\dim(X(\cF)^K)\mid K\in \cF\}\]
    it follows that 
    \[n(H)-w_H-1\ge \dim(X(\cF))\ge \dim(X(\cF)^H).\]
\end{enumerate}

So we apply \cref{lem:attaching} and get one of the two following cases:
\begin{enumerate}
    \item if $w_H\neq 0$, we get a $WH$-homotopy representation $X'(\cF')$ and a $WH$-map \[f':X'(\cF')\xrightarrow{} Y^H,\] with $d(f')(H)=\pm 1$, or
    
    \item if $w_H=0$, we get a $WH$-homotopy representation $X''(\cF')$ and a $WH$-map \[f'':X''(\cF')\xrightarrow{} Y^H,\] with $d(f'')(H)$ prime to $|WH|$. 
\end{enumerate} 

We now examine what to do in case (2). Let $l$ be such that \[l\cdot d(f'')(H)\equiv 1\mod |WH|.\] We want to now apply \cref{lem:4} with $Z=X''(\cF')$ and $G=WH$. To do so we observe that \[\dim(X(\cF)^H)\le n(H)-1\] and $X''(\cF')$ is constructed from $X(\cF)^H$ by attaching free cells \[WH_+\wedge D^l,\ l\le n(H).\] So $X''(\cF')$ is constructed from its $(n(H)-1)$-skeleton by attaching free cells. Then since $l$ is prime to $|WH|$ we can apply \cref{lem:4}. This yields a $WH$-homotopy representation $X'(\cF')$ that is constructed from $X(\cF)^H$ by attaching $WH$-free cells and a $WH$-map \[\varphi: X'(\cF')\xrightarrow{} X''(\cF')\] of degree $k$.
    
Then we use \cref{prop:4} applied to $f''\varphi$ to get a $WH$-map \[f':X'(\mathcal{F}')\xrightarrow{} Y^H\] such that $d(f')(H)=1$ and $f'^K\simeq_{WH}(f''\varphi)^K$ for all $K\supsetneq H$. 

Now, in either case, $X'(\cF')$ was constructed from $X(\cF)^H$ by attaching $WH$-free cells, so we can apply \cref{lem:3}. This gives us a unique $G$-CW-complex $X(\cF')$ that contains $X(\cF)\cup X'(\cF')$ and an extension $f_{\cF'}$ of $f_\cF$ and $f'$. 

We now verify that $X(\cF')$ and $f_{\cF'}$ satisfy i-iii from above. 
\begin{enumerate}[{i.}]
    \item The construction of $X(\cF')$ only added $H$-fixed points to $X(\cF)$, so \[\Iso(X(\cF'))=\Iso(X(\cF))\cup (H)=(\cF\cap\Iso(Y))\cup(H)=\cF'\cap\Iso(Y).\] 
    \item By \cref{rmk:2}, for every $K\in\cF$, \[X(\cF')^K=X(\cF)^K\text{ and } X(\cF')^H=X'(\cF').\] By assumption, $X(\cF)^K\simeq Y^K$ for $K\in\cF$ and since $d(f')(H)=\pm1$, we have a homotopy equivalence $X'(\cF')\simeq Y^H$.
    \item Also by \cref{rmk:2}, for every $K\in\cF$, \[f^K_{\cF'}=f^K_{\cF}\text{ and } f^H_{\cF'}=f'^H.\] So for every $K\in\cF'$, $\deg(f_{\cF'}^K)=\pm1$.
\end{enumerate} 

Note that at each step, when we pick $H\in \Iso(Y)\backslash \cF$, $\cF'$ is constructed to contain the entire conjugacy class of $H$. Therefore, by proceeding  through the finite set of conjugacy classes of isotropy subgroups of $Y$, we end up with a homotopy representation $X$ and a $G$-homotopy equivalence $f:X\xrightarrow{} Y$.
\end{proof}

We started this section with the question: when does there exist a homotopy representation $X$ that is $G$-homotopy equivalent to $Y$? This question is answered with the conditions on $Y$ given in \cref{prop:6}. In the next section we will see that, stably, every equivalence class of generalized homotopy representations contains a representative that satisfies the conditions of \cref{prop:6}.

\section{Proof of the Main Theorem}\label{Grothendieck}

The smash product $X\wedge Y$ of two (generalized) homotopy representations is a (generalized) homotopy representation. The set of $G$-homotopy classes of (generalized) homotopy representations is a commutative monoid under $\wedge$ with $S^0$ as the identity. Denote this monoid by $M(G)$ for homotopy representations and $M'(G)$ for generalized homotopy representations. Then denote the Grothendieck group of homotopy representations for $G$ by $V(G)$ and that for generalized homotopy representations by $V'(G)$. In this section, we will see that the natural map $\rho:V(G)\xrightarrow{} V'(G)$ is 
an isomorphism when $G$ is finite or stably free.

To begin we show that if we smash a generalized homotopy representation with an appropriate representation sphere then the resulting generalized homotopy representation satisfies conditions that allow us to apply \cref{prop:6}.

\begin{prop}\label{prop:7}
Let $G$ be a compact $k$-dimensional Lie group. Suppose $X$ is a generalized homotopy representation. 
Then there exists a finite-dimensional real representation $U$ such that the following hold:
\begin{enumerate}
    \item $\Iso(X\wedge S^U)$ is closed under intersections, and
    \item if $H\in \Iso(X\wedge S^U)$ and $H\subsetneq L$, then \[n(X\wedge S^U)(L)+k< n(X\wedge S^U)(H).\]
\end{enumerate}
\end{prop}
\begin{proof}
    Let $\{(H_i) : 1\leq i\leq m\}\subseteq \Psi(G)$ be the set of all conjugacy classes of isotropy groups of $X$. 
    
    Let $U$ be a finite-dimensional real representation such that each $H_i\in\Iso(U)$ and $\Iso(U)$ is closed under intersections. Such a $U$ exists by \cite[Lemma 2.4]{Bauer1988}. Furthermore, by the construction of $U$, if $H\in \Iso(U)$ and $L\supsetneq H$, then \[\dim(U^L)+k<\dim(U^H).\]    
    
    It follows that $\Iso(X\wedge S^U)$ is closed under intersections and if $H\in \Iso(U)$ and $L\supsetneq H$ then 
    \begin{align*}
        n(X\wedge S^U)(L)+k&=n(X)(L)+n(S^U)(L)+k\\
        &< n(X)(H)+n(S^U)(H)\\
        &= n(X\wedge S^U)(H).\qedhere
    \end{align*}
\end{proof}

\begin{remark}
    The end of the proof of \cite[Lemma 2.4]{Bauer1988} contains a typo which is clarified here. In order to show that $H\cap K$ is an isotropy group, 
    we consider the topological dimension of the union over all $L\supsetneq H\cap K$ of the sets $\{u\in U\mid G_u\in (L)\}$ and note that it is strictly less than $\dim(U^{H\cap K})$. 
    \end{remark}

We now verify that for any generalized homotopy representation $Y$ we can apply \cref{prop:7} to get a $V$ such that $Y\wedge S^V$ satisfies the conditions of \cref{prop:6}.

\begin{prop}\label{prop:8}
    Let $G$ be a compact $k$-dimensional Lie group that is finite or stably free.
    Let $Y$ be a generalized homotopy representation. There exists a representation $V$ of $G$ such that \begin{enumerate}
        \item $n(Y\wedge S^V)(H)\ge w_H+ 3$, for all closed $H\subseteq G$,
        \item if $H\in \Iso(Y\wedge S^V)$, then for all $L\supsetneq H$, \[n(Y\wedge S^V)(H)\ge n(Y\wedge S^V)(L)+w_H+1,\]
        \item $\Iso(Y\wedge S^V)$ is closed under intersections.
    \end{enumerate}
\end{prop}
\begin{proof}
    We begin by smashing $Y$ with $S^{k+3}$ so that \[n(Y\wedge S^{k+3})(G)\ge k+3.\] Let $W$ be the representation of \cref{prop:7} for the generalized homotopy representation $Y\wedge S^{k+3}$. We will show that for $V=\R^{k+3}\oplus W$, $Y\wedge S^V$ satisfies (1)-(3). 
    \begin{enumerate}
        \item  Recall that $m(H)$ is the minimal isotropy subgroup that contains $H$ and \[n(Y\wedge S^V)(H)=n(Y\wedge S^V)(m(H)).\] Since $k\ge w_H$ we have \[n(Y\wedge S^V)(m(H))\ge n(Y\wedge S^V)(G)\ge k+3\ge w_H+3.\] So, \[n(Y\wedge S^V)(H)\ge w_H+3.\]
        \item Let $H\in \Iso(Y\wedge S^V)$ and $L\supsetneq H$. By \cref{prop:7}(2), \begin{align*}
            n(Y\wedge S^V)(H)&\ge n(Y\wedge S^V)(L)+k+1\\
            &\ge n(Y\wedge S^V)(L)+w_H+1.
        \end{align*}
        \item This is \cref{prop:7}(1). \qedhere
    \end{enumerate}
\end{proof}

The next theorem follows directly from \cref{prop:8} and \cref{prop:6}. It will be essential in showing that the map $\rho$ is surjective when $G$ is finite or stably free.

\begin{thm}\label{thm:4}
    Let $G$ be finite or stably free.
    Given a generalized homotopy representation $Y$ and a representation $V$ as in \cref{prop:8}, there exists a homotopy representation $X$ such that \[X\simeq_G Y\wedge S^V.\]
\end{thm}

\begin{lem}
   Let $G$ be finite or stably free. The inclusion \[i:M(G)\xrightarrow{} M'(G)\] induces an injective homomorphism \[\rho:V(G)\xrightarrow{} V'(G).\] 
\end{lem}
\begin{proof}
    Recall that the Grothendieck groups are constructed as equivalence classes of pairs of elements in $M(G)$ (resp. $M'(G)$) with the equivalence relation that $(X_1,X_2)\sim(Y_1,Y_2)$ if there exists $Z$ such that \[X_1\wedge Y_2\wedge Z\simeq_G X_2\wedge Y_1\wedge Z.\] 

    Suppose $\rho(X_1,X_2)=(S^0,S^0)$. This implies that there exists a $Z\in M'(G)$ such that \[i(X_1)\wedge S^0\wedge Z\simeq_G i(X_2)\wedge S^0\wedge Z\]
    Since $i$ is injective this reduces to \[X_1\wedge Z\simeq_G X_2\wedge Z.\] By \cref{thm:4}, there exists a representation $V$ and a homotopy representation $Z_0$ such that $Z\wedge S^V\simeq_G Z_0$. Then \[X_1\wedge Z_0\simeq_G X_2\wedge Z_0,\] so $(X_1,X_2)=(S^0,S^0)$ in $V(G)$ and $\rho$ is injective.
\end{proof}

With the information that $\rho$ is an injective homomorphism we are almost done proving Conjecture 4.5 when $G$ is finite or stably free.

\begin{thm}\label{thm:3}
    Let $G$ be finite or stably free. 
    The homomorphism \[\rho:V(G)\xrightarrow{} V'(G)\] is an isomorphism.
\end{thm}

\begin{proof}
    We need only show that $\rho$ is surjective. Let $Y$ be a generalized homotopy representation. 
    By \cref{thm:4}, there exists a $G$-representation $V$ and a homotopy representation $X$ such that \[\rho([X])=[X]=[Y\wedge S^V].\] Then \[[X]-[S^V]= [Y\wedge S^V]-[S^V]=[Y]\] in $V'(G)$.
Thus we have that \[\rho([X]-[S^V])=[Y]\] and $\rho$ is surjective.
\end{proof}

This completes our proof that $V(G)$ is isomorphic to $V'(G)$ for $G$ finite or stably free. Since $V'(G)$ was already known to be isomorphic to $Pic(hSP^G)$ for all compact Lie groups, we now know that $V(G)\cong Pic(hSp^G)$ for $G$ finite or stably free. This allows us to fit the Picard group of the equivariant stable homotopy category for finite or stably free compact Lie groups into the following short exact sequence, \begin{equation}
    0\xrightarrow{} Pic(A(G))\xrightarrow{} Pic(hSp^G)\xrightarrow{} D(G)\xrightarrow{} 0.
\end{equation} In particular, recall that this means that we have such a short exact sequence for all of the compact Lie groups listed in \cref{orthogonal} and \cref{liegroups}.

\bibliographystyle{amsalpha}
\end{document}